\documentclass{amsart}
\usepackage{graphicx}
\usepackage{amsmath,amssymb}
\usepackage{amsthm}
\usepackage{wrapfig}
\usepackage{enumitem}
\usepackage[hang,small,bf]{caption}
\usepackage[subrefformat=parens]{subcaption}
\usepackage{listings}
\usepackage{ascmac}
\usepackage{mathrsfs}
\usepackage{cite}
\usepackage[legacycolonsymbols]{mathtools}
\usepackage{bm}
\usepackage{comment}

\usepackage{color}

\theoremstyle{plain}
\newtheorem{thm}{Theorem}
\newtheorem*{thm*}{Theorem}

\newtheorem*{prop*}{Proposition}
\newtheorem{lem}{Lemma}
\newtheorem*{lem*}{Lemma}
\newtheorem{exm}{Example}
\newtheorem*{exm*}{Example}

\newtheorem*{pbm*}{Problem}
\newtheorem{rem}{Remark}
\newtheorem*{rem*}{Remark}

\newcommand{\R}{{\mathbb R}}
\newcommand{\Q}{{\mathbb Q}}
\newcommand{\Z}{{\mathbb Z}}
\newcommand{\C}{{\mathbb C}}
\newcommand{\N}{{\mathbb N}}
\newcommand{\K}{{\mathcal K}}
\newcommand{\J}{{\mathcal J}}

\bmdefine{\ba}{a}
\bmdefine{\bb}{b}
\bmdefine{\be}{e}
\bmdefine{\bE}{E}
\bmdefine{\bh}{h}
\bmdefine{\bg}{g}
\bmdefine{\bk}{k}
\bmdefine{\bp}{p}
\bmdefine{\bx}{x}
\bmdefine{\by}{y}
\bmdefine{\bs}{s}
\bmdefine{\bt}{t}
\bmdefine{\bu}{u}
\bmdefine{\bv}{v}
\bmdefine{\bw}{w}

\newcommand{\un}{\underline}
\newcommand{\wi}{\widetilde}

\theoremstyle{definition}

\allowdisplaybreaks

\newcommand{\relmiddle}[1]{\mathrel{}\middle#1\mathrel{}}



\begin{document}
\title[Algebraic independence of non-lacunary power series]{An extension of algebraic independence of special values for non-lacunary power series}
\author{Hajime Kaneko, Satoru Oshima, Takafumi Tsurumaki}
\date{}
\begin{abstract}
We study the algebraic independence of special values of power series $f(\beta^{-1})$, where $\beta$ is a fixed Pisot or Salem number. 
In particular, we consider the case where $f(X)=\sum_{n\geq 0} t(n) X^{w(n)}$ is not a lacunary series and is not assumed to satisfy any special functional equation, such as a Mahler-type functional equation. In our main results, we give a new criterion of the algebraic independence of three values. Applying our main results, we prove that the following three values are algebraically independent: 
    \[
    \sum_{n=3}^{\infty}\lfloor n^{y}\rfloor\beta^{-\lfloor n^{\log \log n}\rfloor},\quad
    \sum_{n=3}^{\infty}\beta^{-\lfloor n^{\log \log n}\rfloor},\quad 
    \sum_{n=1}^{\infty}\beta^{-\lfloor n^{\log n}\rfloor},
    \]
where $y$ is an arbitrary positive real number. Since our criterion is flexible, we have considerable freedom in choosing the coefficients $(t(n))_{n\geq 0}$ and the exponents $(w(n))_{n\geq 0}$.
\end{abstract}
\maketitle

\section{Introduction}
Throughout this paper, let $\N$ denote the set of nonnegative integers and let $\Z_{>0}$ denote the set of positive integers. For a real number $x$, we denote the integral part of $x$ by $\lfloor x \rfloor$. We use the Landau symbols $o,O$ and the Vinogradov symbols $\ll,\gg$ in the usual sense: for two functions $f$ and $g$, the assertions $f\ll g$, $g\gg f$, and $f=O(g)$ mean that there exists a positive constant $C$ such that $|f|\le C|g|$. If a parameter $a$ is involved, then $f\ll_a g$ and $g\gg_a f$ mean that the implicit constant may depend only on $a$. Finally, $f=o(g)$ means that $f/g$ tends to zero.

Throughout this section, we assume that $\bw=(w(m))_{m=0}^{\infty}$ is an eventually strictly sequence of nonnegative integers and that $\alpha\in \C$ is an algebraic number with $0<|\alpha|<1$.  
We survey known results on the arithmetic properties of the values $f(\bw;\alpha)$, where  $f(\bw;X)$ is an infinite series defined by
\[
f(\bw;X)=\sum_{m=0}^{\infty}X^{w(m)}.
\]
After introducing results on the transcendence of the values, we review results on the algebraic independence. 

Bugeaud \cite{MR2378397} posed the problem of whether sufficiently rapid growth of $(w(m))_{m=0}^{\infty}$ forces the value $f(\bw;\alpha)$ to be transcendental.
We call $f(\bw;X)$ a gap series if
\[
\lim_{m\to\infty}\frac{w(m+1)}{w(m)}=\infty,
\]
and a lacunary series if
\[
\liminf_{m\to\infty}\frac{w(m+1)}{w(m)}>1.
\]
Using the $p$-adic Subspace Theorem, Corvaja and Zannier \cite{CorvajaZannier2002} proved that if $f(\bw;X)$ is a lacunary series, then $f(\bw;\alpha)$ is transcendental. For instance, let $x,y$ be real numbers with $x>0$ and $y>1$. Then these two numbers
\[
\sum_{m=0}^{\infty}\alpha^{\lfloor x(m!)\rfloor}
\qquad\text{and}\qquad
\sum_{m=0}^{\infty}\alpha^{\lfloor y^m\rfloor}
\]
are transcendental.

Adamczewski \cite{MR2057021} improved the results above in the case of $\alpha=\beta^{-1}$, where $\beta$ is a Pisot or Salem number. Recall that a Pisot number is an algebraic integer greater than $1$ whose conjugates except itself have absolute value less than $1$. In particular, any integer greater than $1$ is a Pisot number. Moreover, a Salem number is an algebraic integer greater than $1$ whose conjugates except itself have absolute value at most $1$, and at least one conjugate has absolute value $1$. 
Adamczewski \cite{MR2057021} proved that if $(w(m))_{m=0}^{\infty}$ satisfies 
\begin{align}\label{int_ada}
\limsup_{m\to\infty}\frac{w(m+1)}{w(m)}>1,
\end{align}
then, for every Pisot or Salem number $\beta$, the value $f(\bw;\beta^{-1})$ is transcendental.


On the other hand, Mahler \cite{Mahler1929} initiated the method, called Mahler's method, to prove the transcendence of the values of analytic functions satisfying certain functional equations. For instance, the following functions satisfy Mahler type equation: 
\[
\sum_{n=0}^{\infty}z^{k^n}, \quad \prod_{n=0}^{\infty}\left(1-z^{k^n}\right),
\]
where $k$ is an integer greater than 1. \par
On the other hand, the main interest of this paper is the arithmetical properties of real numbers whose $\beta$-representations (infinite sums involving $\beta^{-n}$) in the case where $\beta$ is a Pisot or Salem number. 
In particular, we consider the case where the series do not necessarily satisfy Mahler-type functional equations and do not even satisfy (\ref{int_ada}). 

A fundamental starting point in an integer base $b \ (\geq2)$ is the theorem by Adamczewski and Bugeaud \cite{AdamczewskiBugeaud2007}. They proved that the complexity function $p(b,\xi;n)$ in the $b$-ary expansion of each irrational algebraic number satisfies 
\[\liminf_{n\to\infty}\frac{p(b,\xi;n)}{n}=\infty.\]
As a corollary, any irrational automatic numbers are proved to be transcendental. This result shows that any irrational algebraic numbers cannot have expansions that are too regular.
Analogous results are shown in \cite{AdamczewskiBugeaud2007} for the $\beta$-expansion of algebraic number $\xi\in (0,1)\backslash\Q(\beta)$. 
Thus, the arithmetic study of values of power series naturally extends from ordinary integer bases to Pisot and Salem bases. 
The proof relies on repetitive property of the sequence of coefficients. 
Note that Adamczewski, Bugeaud, and Luca \cite{ABL2008} introduced the notion of a stammering function, which generalizes the repetitive property in \cite{AdamczewskiBugeaud2007}, and obtain partial results for the transcendence at general algebraic points. 

We consider the case where the series are relatively sparse. 
For the proof of our main results, we investigate the sets of nonzero digits in the products of power series, which is mainly denoted in terms of Minkowski sum. We first consider the case where $\beta$ is an integer $b$ greater than 1. 
Erd\H{o}s \cite{1957Erdos} proved for a positive integer $R$ that if 
\begin{align}\label{eqn:erd}
    \limsup_{m\to\infty}\frac{w(m)}{m^R}=\infty,
\end{align}
then $f(\bw;b^{-1})=\sum_{m=0}^{\infty}b^{-w(m)}$ is not an algebraic number of degree at most $R$. In particular, if the above relation holds for any positive integer $R$, then $f(\bw;b^{-1})$ is transcendental, which was proved by Erd\H{o}s and Straus \cite{1957Erdos} with a slightly different method. The result above was rediscovered by Bailey, Borwein, Crandall and Pomerance \cite{2004BBCL}.

Next, we consider the case where $\beta$ is a general Pisot or Salem number. 
By developing the method of Erd\H{o}s and Straus \cite{1957Erdos}, the transcendence result was proved by \cite{Kaneko2015}. 
More precisely, if (\ref{eqn:erd}) holds for any positive integer $R$, then $f(\bw; \beta^{-1})$ is transcendental. 
This result was improved by \cite{Kaneko2016} as follows: assume for a positive integer $R$ that 
\begin{align*}
    \limsup_{m\to\infty}w(m)\cdot\left(\frac{\log m}{m}\right)^R=\infty. 
\end{align*}
Then, 
    the degree of the field extension satisfies $[\Q(\beta,f(\bw;\beta^{-1})):\Q(\beta)]\geq R$, which includes the case where $f(\bw; \beta^{-1})$ is transcendental.

We now introduce known results of the algebraic independence of power series at a single algebraic point. 
Algebraic independence of the values of gap series were studied by many mathematicians. 
Typical criteria for the algebraic independence of gap series are studied in \cite{1982Shiokawa}. For instance, the continuum set 
\begin{align*}
    \left\{
    \left.
    \sum_{n=0}^{\infty}\alpha^{\lfloor x(n!)\rfloor}
    \ \right| \
    x\in \R, x>0
    \right\}
\end{align*}
is algebraically independent. 
For comparison, it is generally difficult to investigate the algebraic independence of the values of general lacunary series. 
Developing Mahler's method, Nishioka \cite{Nishioka1994} proved that the following set is algebraically independent: 
\begin{align*}
    \left\{
    \left. 
    \sum_{n=0}^{\infty} \alpha^{k^n}
    \ \right| \ k=2,3,\ldots
    \right\}
\end{align*}
Tanaka obtained two related generalizations. One concerns the series
\[
\sum_{n=0}^{\infty} \alpha^{\lfloor \rho k^n\rfloor},
\]
where \(\rho>0\) and $k$ is an integer greater than 1 \cite{Tanaka2004}. The other concerns the series $f(\bw;z)$, 
where $(w(n))_{n=0}^{\infty}$ is a linear recurrence sequence \cite{Tanaka1996}.

We now investigate the algebraic independence of the values of $f(\bw; \beta^{-1})=\sum_{n=0}^{\infty} \beta^{-w(n)}$, where $\beta$ is a Pisot or Salem number. 
We assume that $f(\bw;X)$ is relatively sparse and that $f(\bw;X)$ satisfies a sort of uniformity property of nonzero terms. 
In particular, we consider the case where (\ref{eqn:erd}) holds for any positive integer $R$. 
In the case where $\beta$ is an integer $b$ greater than 1, a criterion of algebraic independence was obtained by \cite{Kaneko2012}. 
This shows that algebraic independence can sometimes be obtained directly from the distribution of nonzero digits, without using any Mahler-type functional equation.
Note that the results above mainly treat the case of 
\[
\displaystyle \lim_{m\to\infty}\frac{w(m+1)}{w(m)}=1.
\]
Criterion for algebraic independence for a general Pisot or Salem number $\beta$ are obtained by \cite{Kaneko2019}.
For example, the two values
\[
\sum_{m=1}^{\infty}\beta^{-\lfloor m^{\log m}\rfloor}
\qquad\text{and}\qquad
\sum_{m=3}^{\infty}\beta^{-\lfloor m^{\log\log m}\rfloor},
\]
are algebraically independent. Note that we can show 
\begin{align*}
    &\lim_{m\to\infty}(m+1)^{\log (m+1)}\cdot m^{-\log m}=1,\\
    &\lim_{m\to\infty}(m+1)^{\log\log (m+1)}\cdot m^{-\log\log m}=1
\end{align*}
by the mean value theorem. 
On the other hand, the algebraic independence of the values for unbounded coefficients are also investigated in \cite{Kaneko2017}. 
For instance, the two values
\[
\sum_{m=1}^{\infty}\beta^{-\lfloor m^{\log m}\rfloor}
\qquad\text{and}\qquad
\sum_{m=1}^{\infty}m^r\beta^{-\lfloor m^{\log m}\rfloor},
\]
are algebraically independent, where $r$ is any positive integer.

The purpose of this paper is to combine the results above and prove a new algebraic independence criterion for three values of power series at Pisot or Salem bases. The main criterion is stated in Theorem \ref{mainThm}. As applications, it yields the following representative consequences.
For a real number $x$, let $\log^{+}x=\log\max\{1,x\}$. 
\begin{thm}\label{thm:section1}
    Let $\beta$ be a Pisot or Salem number. Let $(u_i(n))_{n=1}^{\infty}$ be sequences of positive integers for $i=1,2,3$ satisfying the following conditions:
    \begin{enumerate}[
  label={Condition (\arabic*)},
  leftmargin=*,
  align=left
]
        \item There exists a constant $C$ with $0<C<1$ such that, for each $i=1,2,3$,
        \[\log^{+}u_i(n)=o(n^{1-C})\]
        as $n$ tends to infinity.
        \item We have
        \[\lim_{n\to\infty}\frac{u_1(n)}{u_2(n)}=\infty.\]
    \end{enumerate}
    Then we have the following:\\
    (1) For any real numbers $x,y$ with $0<x<y$, the numbers
    \[
    \begin{gathered}
    \sum_{n=1}^{\infty}u_1(n)\beta^{-\lfloor\exp((\log n)^{1+x})\rfloor},\quad 
    \sum_{n=1}^{\infty}u_2(n)\beta^{-\lfloor\exp((\log n)^{1+x})\rfloor},\quad
    \sum_{n=1}^{\infty}u_3(n)\beta^{-\lfloor\exp((\log n)^{1+y})\rfloor}
    \end{gathered}
    \]
    are algebraically independent.\\
    (2) The numbers
    \[
    \begin{gathered}
    \sum_{n=3}^{\infty}u_1(n)\beta^{-\lfloor n^{\log \log n}\rfloor},\quad
    \sum_{n=3}^{\infty}u_2(n)\beta^{-\lfloor n^{\log \log n}\rfloor},\quad 
    \sum_{n=1}^{\infty}u_3(n)\beta^{-\lfloor n^{\log n}\rfloor}
    \end{gathered}
    \]
    are algebraically independent.\\
    (3) Let $r$ be a real number greater than 1. 
    Then the numbers
    \[
    \begin{gathered}
    \sum_{n=1}^{\infty}u_1(n)\beta^{-\lfloor n^{\log n}\rfloor},\quad
    \sum_{n=1}^{\infty}u_2(n)\beta^{-\lfloor n^{\log n}\rfloor},\quad 
    \sum_{n=1}^{\infty}u_3(n)\beta^{-\lfloor r^n\rfloor}
    \end{gathered}
    \]
    are algebraically independent.
    
\end{thm}

\section{Algebraic independence of three real numbers}
In this section, we investigate the algebraic independence of values of power series
$f_1(X), f_2(X), f_3(X) \in \Z[[X]] \setminus \Z[X]$ of the form
\begin{align*}
f_1(X) &= \sum_{m \in S_1} t_1(m) X^m, \quad 
f_2(X) = \sum_{m \in S_1} t_2(m) X^m, \quad 
f_3(X) = \sum_{m \in S_2} t_3(m) X^m,
\end{align*}
where $S_1, S_2 \subset \N$ are infinite sets, $t_1(m),t_2(m)\in\Z_{>0}$ for $m\in S_1$, and $t_3(m)\in\Z_{>0}$ for $m\in S_2$.

We introduce the set-theoretic notation used throughout the paper. Let $\mathcal{A}$ and $\mathcal{B}$ be nonempty subsets of $\N$. For each nonnegative real number $R$, set
\[
\lambda(\mathcal{A};R)=\operatorname{Card}([0,R)\cap\mathcal{A}).
\]
When $R>\min \mathcal{A}$, put
\[
\theta(\mathcal{A};R)=\max([0,R)\cap\mathcal{A}).
\]
We define the Minkowski sum $\mathcal{A}+\mathcal{B}$ by
\[
\mathcal{A}+\mathcal{B}:=\{a+b\mid a\in\mathcal{A},\ b\in\mathcal{B}\}.
\]
For each positive integer $k$, set $k\mathcal{A}:=\underbrace{\mathcal{A}+\cdots+\mathcal{A}}_{k\ \text{times}}$, and put $0\mathcal{A}:=\{0\}$.

We now state our criterion for algebraic independence. For convenience, put $t_1(m)=t_2(m)=0$ for $m\in \N\setminus S_1$, and put $t_3(m)=0$ for $m\in \N\setminus S_2$. 
\begin{thm}\label{mainThm}
We assume that the following conditions.
\begin{itemize}
\item[\rm{(1)}] There exists a constant $C_1>1$ such that, for all sufficiently large $R$, we have $[R,C_1R)\cap S_2\ne\emptyset$.
\item[\rm{(2)}] 
For every positive real number $\varepsilon$, we have, for $j=1,2$,
\begin{align*}
\lambda(S_j;R)&=o(R^\varepsilon)
\end{align*}
as $R$ tends to infinity.
\item[\rm{(3)}] For any nonnegative integers $a_1,a_2$, there exists a positive constant $C_2=C_2(a_1,a_2)$, depending only on $a_1,a_2$ such that, for all $R\ge C_2$,
\[
R-\theta((1+a_1)S_1;R)<\frac{R}{\lambda(S_1;R)^{a_1}\lambda(S_2;R)^{a_2}}.
\]
\item[\rm{(4)}] There exists a constant $C_3$ with $0<C_3<1$ such that, for each $i=1,2,3$,
\[
\log^{+}t_i(n)=o(n^{1-C_3})
\]
as $n$ tends to infinity.
\item[\rm{(5)}] We have
\[
\lim_{\substack{n\to\infty\\ n\in S_1}}\frac{t_1(n)}{t_2(n)}=\infty.
\]
\end{itemize}
Then, for every Pisot or Salem number $\beta$, the three values $f_1(\beta^{-1}), f_2(\beta^{-1}), f_3(\beta^{-1})$ are algebraically independent.
\end{thm}

We give the proof of Theorem \ref{thm:section1} by using Theorem \ref{mainThm}.
\begin{proof}[Proof of Theorem \ref{thm:section1}]
Conditions (4) and (5) in Theorem \ref{mainThm} follow from Conditions (1) and (2) in Theorem \ref{thm:section1}. 
We denote the numbers in Theorem \ref{thm:section1} by
\begin{align*}
    \sum_{n\in S_1} u_1(n)\beta^{-n}
    , \quad
    \sum_{n\in S_1} u_2(n)\beta^{-n}
    , \quad
    \sum_{n\in S_2} u_3(n)\beta^{-n}
\end{align*}
where $S_1,S_2\subset \N$. 
Putting 
\begin{align*}
    S_2=:\{w(0)<w(1)<\cdots\},
\end{align*}
we see 
\begin{align*}
    \limsup_{m\to\infty}\frac{w(m+1)}{w(m)}<\infty.
\end{align*}
Thus, Condition (1) in Theorem \ref{mainThm} holds. Conditions (2) and (3) can be proved in the same way as the method of \cite{Kaneko2012,Kaneko2019}. 
\end{proof}

\section{Proof of Theorem \ref{mainThm}}
We introduce notation for the proof of Theorem \ref{mainThm}. Let $\mathbb{N}^{0}:=\{0\}$ and for $r \in \mathbb{N} $ and $\bm{l}=\left(l_{1}, \ldots, l_{r}\right) \in \mathbb{N}^{r}$, we define
\begin{equation*}
\begin{aligned}
|\bm{l}| & := \begin{cases}0 & (r=0), \\
l_{1}+\cdots+l_{r} & (r \geq 1),\end{cases} \\
t_i(\bm{l}) & := \begin{cases}1 & (r=0), \\
t_i\left(l_{1}\right) \cdots t_i\left(l_{r}\right) & (r \geq 1). \end{cases}
\end{aligned}
\end{equation*}
Moreover, let $\underline{X}:=(X_1, X_2, X_3)$ and let $\underline{X}^{\bm{k}}:=X_1^{k_1}X_2^{k_2}X_3^{k_3}$ for $\bm{k}=(k_1,k_2,k_3)$. 

We define the hybrid order $>_{\mathrm{hyb}}$ on $\mathbb{N}^{3}$ to classify monomials $\underline{X}^{\bm{k}}$. 
First, let $>_{\mathrm{lex}}$ (resp. $>_{\mathrm{gr}}$) be the lexicographic order (resp. graded lexicographic order) on $\N^2$.
That is, for ${\bm p} = (p_1,p_2),\ {\bm p'} = (p_1',p_2') \in \mathbb{N}^{2}$ with ${\bm p} \neq {\bm p'}$,\ we define ${\bm p}>_\mathrm{lex} {\bm p'}$ if and only if 
$p_1 > p_1'$, or $p_1=p_1'$ and $p_2 > p_2'$. Similarly, ${\bm p}>_\mathrm{gr} {\bm p'}$ holds if and only if $p_1+p_2 > p_1'+p_2'$, or $p_1+p_2=p_1'+p_2'$ and $p_1 > p_1'$. 
For ${\bm k}= (k_1,k_2,k_3),\ {\bm k'} = (k_1',k_2',k_3')\in \mathbb{N}^{3}$ with ${\bm k} \neq {\bm k'}$,\ we define ${\bm k}>_{\mathrm{hyb}} {\bm k'}$ if and only if 
$(k_1+k_2 ,k_3) >_\mathrm{lex} (k_1'+k_2' ,k_3')$, or $(k_1+k_2 ,k_3)=(k_1'+k_2' ,k_3')$ and $(k_1,k_2) >_\mathrm{gr} (k_1',k_2')$. Note that $>_{\mathrm{hyb}}$ is a total order on $\N^3$. 
For instance, we have 
\[
(3,2,0) >_{\mathrm{hyb}} (2,3,0) >_{\mathrm{hyb}} (2,2,7).
\]

In what follows, we write $\xi_j=f_j(\beta^{-1})$ for $j=1,2,3$. For any $\ell_1,\ell_2,\ell_3 \in \mathbb{Z}$, the algebraic independence of $\xi_1,\xi_2,\xi_3$ is equivalent to the algebraic independence of $\xi_1+\ell_1,\xi_2+\ell_2,\xi_3+\ell_3$. Thus, without loss of generality, we may take $0 \in S_1 \cap S_2$ and $t_1(0)=t_2(0)=t_3(0)=1$. 

We shall show that $P(\underline{\xi})=P(\xi_1,\xi_2,\xi_3)\ne0$ for every polynomial $P(\underline{X})\in\mathbb{Z}[X_1,X_2,X_3]\setminus\{0\}$. Since $\xi_1\xi_2\xi_3\ne0$, it is enough to prove this assertion after replacing $P$ by $X_1X_2X_3P$ if necessary. Hence we may assume that $X_1X_2X_3$ divides $P(\underline{X})$. Let $D$ be the total degree of $P$. 
We write $P(\underline{X})$ as 
\begin{equation}\label{eq:poly1}
P(\underline{X})= A_{\bm 0}+\sum_{{\bm{k}} \in \Gamma} A_{\bm{k}} \underline{X}^{\bm{k}},
\end{equation}
where $\Gamma$ is a finite nonempty subset of $\N^3\setminus \{\bm 0\}$ and $A_{\bm k}$ is an integer for every $\bm{k}\in \Gamma\cup\{\bm 0\}$.

Let ${\bm g} = (g_1,g_2,g_3)$ be the greatest element of $\{\bm k\in \Gamma\mid A_{\bm k}\ne 0\}$ with respect to $>_{\mathrm{hyb}}$. 
If necessary, by adding suitable terms of the form $0\cdot \underline{X}^{\bm k}$ with ${\bm k}\in \N^3 \setminus\{\bm 0\}$ to (\ref{eq:poly1}), 
we may assume that 
\begin{align*}
    \Gamma=
    \left \{{\bm{k}} \in \mathbb{N}^{3} \relmiddle| \ k_1, k_2, k_3 \leq D, \ {\bm{k}} \leq_{\mathrm{hyb}}{\bm{g}} \right\} \setminus \{{\bm 0}\}
\end{align*}
because there exist only finitely many $\bm{k}=(k_1,k_2,k_3)\in \N^3 \setminus\{\bm 0\}$ satisfying $\ k_1, k_2, k_3 \leq D$ and  ${\bm{k}} \leq_{\mathrm{hyb}}{\bm{g}}$. 
Without loss of generality,\ we may assume that $A_{\bm g}\geq 1$. 
Since $X_1X_2X_3$ divides $P(\underline{X})$, we see that $g_j\geq 1$ for $j=1,2,3$, and $D\geq 3$. 
Moreover, we define $\Gamma_0,\ \Gamma_1$ and $\Gamma_2$ as follows.
\begin{equation*}
 \begin{split}
  \Gamma_0 :=& \{(k_1,k_2,k_3) \in \Gamma \mid k_1+k_2=g_1+g_2, \ k_3=g_3 \} \setminus \{ {\bm g} \}, \\
  \Gamma_1 :=& \{(k_1,k_2,k_3) \in \Gamma \mid k_1+k_2=g_1+g_2, \ k_3<g_3 \}, \\
  \Gamma_2 :=& \{(k_1,k_2,k_3) \in \Gamma \mid k_1+k_2<g_1+g_2 \}.\\
 \end{split}
\end{equation*}
Then we have the following disjoint union:
\[
\Gamma = \{{\bm g} \} \sqcup \Gamma_0 \sqcup \Gamma_1 \sqcup \Gamma_2.
\]
\begin{rem}
The greatest elements of $(\Gamma_0, >_{\mathrm{hyb}})$, $(\Gamma_1, >_{\mathrm{hyb}})$, and $(\Gamma_2, >_{\mathrm{hyb}})$ are
\[
{\bm k_{0}}:=(g_1-1,g_2+1,g_3),\quad
{\bm k_{1}}:=(g_1+g_2,0,g_3-1),\quad
{\bm k_{2}}:=(g_1+g_2-1,0,D),
\]
respectively. In particular, $\Gamma_0,\Gamma_1,\Gamma_2\ne\emptyset$.
\end{rem}
Now, we define the positive constant $C_4$ as
\begin{equation}\label{eq:C4}
C_4:= \max \bigg\{1,\ \frac{2}{A_{\bm g}}\bigg(1+ \sum_{{\bm k} \in \Gamma_0} |A_{\bm k}| \bigg) \bigg\}.
\end{equation}
Condition\,(5) implies that there exists $C_5 \in \mathbb{Z}_{>0}$ such that for any $n \geq C_5$ with $n \in S_1$, 
\begin{equation}\label{eq:con5}
t_1(n) >C_4t_2(n).
\end{equation}
For $i=1,2$, we define $\widetilde{S_i} :=\big([C_5, \infty) \cap S_i \big) \cup \{ 0\}$.\ Since $0 \in S_1\cap S_2$,\ we have $\widetilde{S_i} \subset S_i$.
From (\ref{eq:con5}),\ for any $n \in \widetilde{S_1} \setminus \{ 0\}$,
\begin{equation} \label{eq:tt}
t_1(n) >C_4t_2(n).
\end{equation}
Now, we define $\eta_1,\, \eta_2$ and $\eta_3$ as follows:
\begin{equation*}
\eta_1 := \sum_{m \in \widetilde{S_1}} t_1(m) \beta^{-m}, \quad 
\eta_2 := \sum_{m \in \widetilde{S_1}} t_2(m)\beta^{-m}, \quad 
\eta_3 := \sum_{m \in \widetilde{S_2}} t_3(m)\beta^{-m}.
\end{equation*}
Then $\beta^{C_5}(\xi_i - \eta_i) \in \mathbb{Z}[\beta]$ for $i=1,2,3$ and so 
\begin{equation*} 
\begin{split}
\beta^{C_5 D}P(\underline{\xi})&=\beta^{C_5D}A_{\bm 0}+\sum_{{\bm k} \in \Gamma}\beta^{C_5 D}A_{\bm k} \underline{\xi}^{\bm k} \\
&= \beta^{C_5D}A_{\bm 0}
  +\sum_{{\bm k} \in \Gamma}\beta^{C_5(D- |{\bm k}|)} A_{\bm k}
  \prod_{i=1}^3 (\beta^{C_5} (\xi_i-\eta_i) + \beta^{C_5} \eta_i)^{k_i} \in (\mathbb{Z}[\beta])[\underline{\eta}].
\end{split}
\end{equation*}
We set
\[
Q(\underline{X}) :=\beta^{C_5D}A_{\bm 0}+\sum_{{\bm k} \in \Gamma}\beta^{C_5(D- |{\bm k}|)} A_{\bm k} \prod_{i=1}^3 (\beta^{C_5} (\xi_i-\eta_i) + \beta^{C_5} X_i)^{k_i}.
\]
Then we can write $Q(\un{X})$ as 
\[Q(\un{X})=B_{\bm 0}+\sum_{{\bm k} \in \Gamma} B_{\bm k} \underline{X}^{\bm k},\]

\noindent where $B_{\bm k} \in \mathbb{Z}[\beta]$ for any ${\bm k} \in \Gamma \cup \{{\bm 0}\}$.\ Moreover, we have 
\begin{equation*}
Q(\underline{\eta})=\beta^{C_5D}P(\underline{\xi}).
\end{equation*}
\begin{lem} \label{lem:AB}
For any ${\bm k} \in \Gamma_0 \cup \{{\bm g}\}$,
\begin{equation}\label{eq:AB}
B_{\bm k} = \beta^{C_5 D} A_{\bm k}.
\end{equation}
\end{lem}
\begin{proof}
Recall for any $\bk=(k_1,k_2,k_3)\in \Gamma\cup \{\bm 0\}$ that $\bk\in \Gamma_0\cup\{\bg\}$ if and only if 
\[(k_1+k_2,k_3)=(g_1+g_2,g_3).\] 
Considering the sum of monomials $\un{X}^{\bk}$ in $Q(\un{X})$ satisfying the condition above, we get 
\begin{equation*}
\begin{split}
\sum_{{\bm k} \in \Gamma_0 \cup \{{\bm g}\}}B_{\bm k} \underline{X}^{\bm k} &=\sum_{{\bm k} \in \Gamma_0 \cup \{{\bm g}\}}\beta^{C_5(D- |{\bm k}|)} A_{\bm k} \prod_{i=1}^3 \beta^{C_5k_i}X_i^{k_i} \\
&=\sum_{{\bm k} \in \Gamma_0 \cup \{{\bm g}\}}\beta^{C_5D} A_{\bm k}\underline{X}^{\bm k}.
\end{split}
\end{equation*}
\end{proof}
For any monomial $\underline{\eta}^{\bm k}$ with ${\bm k}= (k_1, k_2, k_3) \in \Gamma $, we see that
\begin{equation*}
\begin{split}
&\underline{\eta}^{{\bm k}} 
= \left(\sum_{m \in \widetilde{S_1}} t_1(m) \beta^{-m}\right)^{k_1}
  \left(\sum_{m \in \widetilde{S_1}} t_2(m)\beta^{-m}\right)^{k_2}
  \left(\sum_{m \in \widetilde{S_2}} t_3(m)\beta^{-m}\right)^{k_3} \\
&=\left(\sum_{\bm{m}_{1} \in \widetilde{S_1}^{k_1}} t_1(\bm{m}_{1}) \beta^{-|\bm{m}_{1}|}\right)
  \left(\sum_{\bm{m}_{2} \in \widetilde{S_1}^{k_2}} t_2(\bm{m}_{2})\beta^{-|\bm{m}_{2}|}\right)
  \left(\sum_{\bm{m}_{3} \in \widetilde{S_2}^{k_3}} t_3(\bm{m}_{3})\beta^{-|\bm{m}_{3}|}\right) \\
&=\sum_{\substack{\bm{m}_{1} \in \widetilde{S_1}^{k_1},\ \bm{m}_{2} \in \widetilde{S_1}^{k_2}\\
\bm{m}_{3} \in \widetilde{S_2}^{k_3}}}
 t_1(\bm{m}_{1})t_2(\bm{m}_{2})t_3(\bm{m}_{3})
 \beta^{-(|\bm{m}_{1}|+|\bm{m}_{2}|+|\bm{m}_{3}|)} .
\end{split}
\end{equation*}
Set 
\[
\Xi({\bm k},m) := \{ ({\bm a}, {\bm b}) \in \widetilde{S_1}^{k_1+k_2}\times \widetilde{S_2}^{k_3} \mid m= |{\bm a}| + |{\bm b}| \}.
\]
For any ${\bm a}=(a(1), a(2), \cdots a(k_1+k_2)) \in \widetilde{S_1}^{k_1+k_2}$, let 
\begin{align}\label{eqn:vector}
{\bm a({\bm k}, 1)}:=(a(1), \cdots , a(k_1)), \ 
{\bm a({\bm k}, 2)}:=(a(k_1+1), \cdots , a(k_1+k_2)).
\end{align}
Putting 
\[
\rho({\bm k} ;m) :=\sum_{({\bm a}, {\bm b}) \in \Xi({\bm k},m)} t_1({\bm a({\bm k}, 1)}) t_2({\bm a({\bm k}, 2)}) t_3({\bm b}) \in \N,
\]
we have 
\[\un{\eta}^{\bk}=\sum_{m=0}^\infty \rho({\bm k} ;m) \beta^{-m}.\]

\begin{lem} \label{lem:Minkowski}
Let $m \in \mathbb{N}$ and let ${\bm k}= (k_1, k_2, k_3) \in \Gamma$.\ Then $\rho({\bm k}; m) \geq 1$ if and only if $m \in (k_1+k_2) \widetilde{S_1}+k_3 \widetilde{S_2}$.
\end{lem} 
\begin{proof}
If $n\in\widetilde{S_1}$, then $t_1(n)$ and $t_2(n)$ are positive, and if $n\in\widetilde{S_2}$, then $t_3(n)$ is positive. Thus the lemma follows because $\rho({\bm k};m)>0$ if and only if $\Xi({\bm k},m)$ is not an empty set.
\end{proof}
\noindent Let $R$ be a nonnegative integer.\ Then
\begin{equation*}
\begin{aligned}
\beta^{R+C_5 D} P(\underline{\xi})&=\beta^{R}Q(\underline{\eta})
=B_{\bm 0}\beta^{R}+\beta^{R}\sum_{{\bm k} \in \Gamma} B_{\bm k} \underline{\eta}^{\bm k} \\
&=B_{\bm 0}\beta^{R}+\sum_{{\bm k} \in \Gamma} B_{\bm k} \sum_{m=-R}^\infty \beta^{-m}\rho({\bm k} ;m+R).
\end{aligned}
\end{equation*}
\noindent We put 
\begin{equation*}
\begin{split}
& Y_R:= \sum_{\bm{k} \in \Gamma} B_{\bm{k}}\sum_{m=1}^\infty \beta^{-m}\rho({\bm k} ;m+R), \\
& Z_R:= B_{\bm 0}\beta^{R}+\sum_{\bm{k} \in \Gamma} B_{\bm{k}}\sum_{m=-R}^{0} \beta^{-m}\rho({\bm k} ;m+R).
\end{split}
\end{equation*}
Then
\begin{equation}\label{eq:PandYRZR}
\beta^{R+C_5 D} P(\underline{\xi})=Y_R+Z_R.
\end{equation}
Note that $Z_R \in \mathbb{Z}[\beta]$.\ We prove Theorem \ref{mainThm} by evaluating $Y_R$ and $Z_R$. 
Let $\mathcal{C}:=C_3/2$. 
First, we show for any sufficiently large integer $R$ that
$$
Z_R=0 \quad \text{or} \quad |Z_R| \geq \beta^{-R^{1-2\mathcal{C}}}
$$
in Lemma \ref{lem:algintZRineq}.\ Next, we show in Lemma \ref{lem:I1YR} that there exists an arbitrarily large integer $R$ satisfying 
$$
0<Y_R< \beta^{-R^{1-2\mathcal{C}}},
$$
which implies $P(\underline{\xi})\neq0$ by (\ref{eq:PandYRZR}). 
We first estimate $Z_R$, using the following lemma. 
\begin{lem} \label{lem:logrho}
For any ${\bm k} \in \Gamma$,
$$
\log^+ \rho({\bm k}; m)  = o(m^{1-2\mathcal{C}})
$$
as $m$ tends to infinity.
\end{lem}

\begin{proof}
Let ${\bm k} =(k_1,k_2,k_3)\in \Gamma$. Since $\Xi({\bm k},m) \subset \{0,1,\cdots , m \}^{k_1+k_2+k_3}$, we have
\begin{equation}
\begin{split}
\rho({\bm k}; m) &\leq (m+1)^{k_1+k_2+k_3} \bigg( \max_{\substack{i=1,2,3 \\ 0 \leq n \leq m}}t_i(n) \bigg)^{k_1+k_2+k_3}\\
&\leq (m+1)^{3D} \bigg( \max_{\substack{i=1,2,3 \\ 0 \leq n \leq m}}t_i(n) \bigg)^{3D} \label{eq:rhomax}.
\end{split}
\end{equation}
Let $\varepsilon$ be an arbitrary positive real number.  
By Condition\,(4), if $n$ is sufficiently large depending only on $\varepsilon$, then, for $i=1,2,3,$ 
\begin{equation*}
    \log^+ t_i(n)<\frac{\varepsilon}{9D} n^{1-2\mathcal{C}}.
\end{equation*}
In particular, if $m$ is sufficiently large depending only on $\varepsilon$, then we have 
\[\log^{+}\left(\max_{\substack{i=1,2,3 \\ 0\leq n\leq m}} t_i(n)\right)\leq \frac{\varepsilon}{9D} m^{1-2\mathcal{C}},\]
and so, by (\ref{eq:rhomax}), 
\[
\log^{+} \rho(\bk;m)< \varepsilon m^{1-2\mathcal{C}}.
\]
\end{proof}
\noindent Now we estimate $Z_R$.
\begin{lem}\label{lem:algintZRineq}
For any sufficiently large integer $R$,
$$
Z_R=0 \quad \text{or} \quad |Z_R| \geq \beta^{-R^{1-2\mathcal{C}}}.
$$
\end{lem}
\begin{proof}
Put $d:=\deg\, \beta$.\ 
We may assume that $d\geq 2$. In fact, if $d=1$, then Lemma \ref{lem:algintZRineq} is clear because $Z_R$ is a rational integer. 
Let $\pi_i: \mathbb{Q}(\beta) \to \mathbb{C}$ be the conjugate map of $\beta$ for $i=1,\cdots,d$,\ where $\pi_1=\mathrm{id}|_{\mathbb{Q}(\beta)}$.\ Moreover, define $\beta_i:=\pi_i (\beta)$ for $i =1, \cdots, d$.\ 
Note that $|\beta_i| \leq 1$ for $i=2, \cdots ,d$ because $\beta$ is a Pisot number or a Salem number.\ Hence, there exists a positive constant $C_6>0$ such that
\begin{equation*}
\begin{split}
|\pi_i(Z_R)| &\leq |\pi_i(B_{\bm 0})||\beta_i|^{R}+\sum_{{\bm k} \in \Gamma} |\pi_i(B_{\bm k})| \sum_{m=0}^R |\beta_i|^{m} \rho({\bm k}; -m+R) \\
& \leq C_6+C_6 (R+1)\sum_{{\bm k} \in \Gamma} \max_{0 \leq m \leq R}\rho({\bm k};m).
\end{split}
\end{equation*}
By Lemma \ref{lem:logrho},\ for $i= 2,\cdots, d$,
\begin{equation*}
\log^{+}|\pi_i(Z_R)| = o(R^{1-2\mathcal{C}})
\end{equation*}
as $R$ tends to infinity.\ Therefore,\ for any sufficiently large $R$,
\begin{equation*}
\sum_{i=2}^d \log^{+}|\pi_i(Z_R)| \leq R^{1-2\mathcal{C}} \log \beta.
\end{equation*}
Thus we have
\begin{equation*}
\prod_{i=2}^d | \pi_i(Z_R) | \leq \beta^{R^{1-2\mathcal{C}}}.
\end{equation*}
Now assume that $Z_R\neq 0$.\ Since $Z_R \in \mathbb{Z}[\beta]$,\ we have
\begin{equation*}
1\leq | Z_R| \prod_{i=2}^d |\pi_i(Z_R)|.
\end{equation*}
Therefore, for any sufficiently large $R$,\ we obtain
\begin{equation*}
|Z_R| \geq \beta^{-R^{1-2\mathcal{C}}}.
\end{equation*}
\end{proof}
\noindent In what follows, we estimate $Y_R$.\ First, we define the set $V$ as 
\begin{equation*}
V:= \bigcup_{\substack{
  {\bm k} \in \Gamma_1 \sqcup \Gamma_2\\
   \bk=(k_1,k_2,k_3)
 }
}
\left(
(k_1+k_2) \widetilde{S_1}+k_3 \widetilde{S_2}
\right).
\end{equation*}
\begin{lem} \label{lem:V1}
For any nonnegative integer $m$ with $m \notin V$,
\begin{equation} \label{eq:V2}
\sum_{{\bm k }\in \Gamma \setminus \{ {\bm g} \}} |B_{\bm k} | \rho ({\bm k} ; m) \leq \frac{1}{2} B_{\bm g} \rho({\bm g};m).
\end{equation}
\end{lem}
\begin{proof}
Let $m \notin V$.\ Then we have $m \notin (k_1+k_2) \widetilde{S_1}+k_3 \widetilde{S_2}$ for all ${\bm k}=(k_1, k_2, k_3) \in \Gamma_1 \sqcup \Gamma_2$.\ From Lemma \ref{lem:Minkowski},\ $\rho({\bm k} ;m)=0$ for any ${\bm k} \in \Gamma_1 \sqcup \Gamma_2$.\ Since $\Gamma \setminus \{{\bm g} \} = \Gamma_0 \sqcup \Gamma_1 \sqcup \Gamma_2$,\ we have
\begin{equation}\label{eq:22}
\begin{aligned}
\sum_{{\bm k }\in \Gamma \setminus \{ {\bm g} \}} |B_{\bm k} | \rho ({\bm k} ; m)
&=\sum_{{\bm k }\in \Gamma_0} |B_{\bm k} | \rho ({\bm k} ; m).
\end{aligned}
\end{equation}
Without loss of generality, we may assume that $m \in (g_1+g_2)\widetilde{S_1}+g_3\widetilde{S_2}$, and so $\Xi(\bg,m)$ is nonempty. 
 In fact,\ if $m \notin (g_1+g_2)\widetilde{S_1}+g_3\widetilde{S_2}$,\ then we have $\rho({\bm k} ;m)=0$ for any ${\bm k} \in \Gamma_0 \cup \{{\bm g} \}$ by Lemma \ref{lem:Minkowski}.\ Thus both sides of (\ref{eq:V2}) are $0$.\ We now check that for any ${\bm k} \in \Gamma_0$,\ 
\begin{equation}\label{eq:3333}
\rho ({\bm k} ; m) <\frac{1}{C_4}\rho ({\bm g} ; m).
\end{equation}
For any ${\bm k}=(k_1,k_2,k_3) \in \Gamma_0$ and $m \notin V$, we have 
\[
\Xi(\bk,m)=\Xi(\bg,m).
\]
In the rest of the proof of Lemma \ref{lem:V1}, we use notation (\ref{eqn:vector}). 
We take $({\bm a}, {\bm b}) \in \Xi({\bm k},m)$. Recall that $k_1+k_2=g_1+g_2,\ k_3=g_3$ and $k_1<g_1$. 
Moreover, we have $a(1), \cdots, a(k_1+k_2)>0$. 
In fact, if there exists an $i$ with $1 \leq i \leq g_1+g_2$ satisfying $a(i)=0$,\ we get 
\[
m=\left( \sum_{\substack{1 \leq j \leq g_1+g_2 \\ j\neq i}} a(j) \right) + |{\bm b}| \in V,
\]
a contradiction.\ Therefore,\ using (\ref{eq:tt}),\ we obtain
\begin{equation*}
\begin{split}
&t_1({\bm a({\bm k}, 1)})t_2({\bm a({\bm k}, 2)}) t_3({\bm b}) \\
= &\left(\displaystyle \prod_{x=1}^{k_1}t_1(a(x))\right)\left(\displaystyle \prod_{y=k_1+1}^{g_1}t_2(a(y))\right)\left(\displaystyle \prod_{z=g_1+1}^{g_1+g_2}t_2(a(z))\right)t_3({\bm b}) \\
< & \dfrac{1}{{C_4}^{g_1-k_1}}\left(\displaystyle \prod_{x=1}^{k_1}t_1(a(x))\right)\left(\displaystyle \prod_{y=k_1+1}^{g_1}t_1(a(y))\right)\left(\displaystyle \prod_{z=g_1+1}^{g_1+g_2}t_2(a(z))\right)t_3({\bm b}) \\
\leq & \dfrac{1}{C_4}t_1({\bm a({\bm g}, 1)})t_2({\bm a({\bm g}, 2)})t_3({\bm b}).
\end{split}
\end{equation*}
Then
\begin{equation*}
\begin{split}
\rho({\bm k} ;m) & =\sum_{({\bm a} ,{\bm b}) \in \Xi({\bm k},m)} t_1({\bm a({\bm k}, 1)})t_2({\bm a({\bm k}, 2)})t_3({\bm b}) \\
& < \dfrac{1}{C_4}\sum_{({\bm a} ,{\bm b}) \in \Xi({\bm k},m)} t_1({\bm a({\bm g}, 1)})t_2({\bm a({\bm g}, 2)})t_3({\bm b}) \\
& = \dfrac{1}{C_4}\sum_{({\bm a} ,{\bm b}) \in \Xi({\bm g},m)} t_1({\bm a({\bm g}, 1)})t_2({\bm a({\bm g}, 2)})t_3({\bm b}) = \dfrac{1}{C_4}\rho({\bm g} ;m),
\end{split}
\end{equation*}
which implies (\ref{eq:3333}).\ By (\ref{eq:C4}) and (\ref{eq:AB}),\ we get
\begin{equation} \label{eq:BC}
\sum_{{\bm k }\in \Gamma_0} |B_{\bm k} |<\dfrac{1}{2}C_4B_{\bm g}.
\end{equation}
Therefore, by (\ref{eq:22}),\ (\ref{eq:3333}) and (\ref{eq:BC}),\ we obtain
\begin{equation*}
\sum_{{\bm k }\in \Gamma_0} |B_{\bm k} | \rho ({\bm k} ; m)\leq \dfrac{1}{C_4}\sum_{{\bm k }\in \Gamma_0} |B_{\bm k} | \rho ({\bm g} ; m)<\frac{1}{2} B_{\bm g} \rho({\bm g};m).
\end{equation*}
\end{proof}
\noindent Define 
\[{\bm e}:= (0,g_1+g_2-1,1+D), \quad {\bm e'}:= (g_1+g_2-1,1+D).\] 
For simplicity, we set, for $i=1,2$,
\begin{equation*}
\begin{aligned}
&\lambda_i(R):= \lambda(\widetilde{S_i};R), \\
&\theta(R) := \theta((g_1+g_2)\widetilde{S_1};R),
\end{aligned}
\end{equation*}
and
\[
\underline{\lambda}(N)^{\bm{k}}:=\lambda_1(N)^{k_1+k_2}\lambda_2(N)^{k_3}
\]
for each ${\bm k}=(k_1, k_2, k_3) \in \mathbb{N}^{3}$.\ Note that $\displaystyle \lim_{R \rightarrow \infty} \dfrac{\lambda(\widetilde{S_i};R)}{\lambda(S_i;R)}=1$ for $i=1,2$. 
We shall check that $\wi{S_1}$ and $\wi{S_2}$ satisfy analogues of Conditions (1), (2) and (3) as follows:
\begin{itemize}
\item[\rm{(1')}] 
For any sufficiently large  $R$, we have $[R, C_1R) \cap \widetilde{S_2} \neq \emptyset$, where $C_1$ is defined in Condition (1).
\item[\rm{(2')}] 

For every positive real number $\varepsilon$, we have, for $j=1,2$
\begin{align*}
\lambda_j(R)
=o(R^\varepsilon)
\end{align*}
as $R$ tends to infinity.
\item[\rm{(3')}] For any nonnegative integers $a_1,a_2$, 
there exists a positive constant $C_2' =C_2'(a_1,a_2)$, depending only on $a_1,a_2$ 
such that for all $R \geq C_2'$, we have 
$$R-\theta((1+a_1)\widetilde{S_1} ;R)<\frac{R}{\lambda_1(R)^{a_1}\lambda_2(R)^{a_2}}.$$
Since $\wi{S_2}$ is an infinite set, we assume that $C_2'\in \wi{S_2}$ for convenience. 
In particular, we have $C_2'\geq C_5$ by $\wi{S_2}\cap(0,\infty)\subset [C_5,\infty)$.
\end{itemize}
Condition (1') holds for any sufficiently large $R \geq C_5$ by Condition (1).\ Condition (2') is also true by Condition (2). 
We prove Condition (3').\ We check that for any positive integer $R$,
\begin{equation} \label{eq:theta3}
\theta((a_1+1)S_1;R)-(a_1+1)C_5 < \theta((a_1+1)\widetilde{S_1};R) \leq \theta((a_1+1)S_1;R).
\end{equation}
The second inequality of (\ref{eq:theta3}) follows from  $(a_1+1)\widetilde{S_1} \subset (a_1+1)S_1$. \\
We write $\theta((a_1+1)S_1;R)=x_1+ \cdots+x_{a_1+1}$,\ where $x_1, \cdots, x_{a_1+1} \in S_1$. 
If necessary, arranging $x_i$ ($i=1,\ldots,a_1+1$), we may assume the following:  
\begin{equation*}
\begin{aligned}
& x_1\leq x_2 \leq \cdots \leq x_{a_1+1}, \\
& x_1, \cdots ,x_t \geq C_5, \\
& x_{t+1}, \cdots ,x_{a_1+1} < C_5,
\end{aligned}
\end{equation*}
where $t=0$ if there is no $i$ such that $x_i \geq C_5$. 
For $i=1, \cdots, a_1+1$,\ we define $\widetilde{x}_i \in \wi{S_1}$ as
\begin{equation*}
\begin{aligned}
\widetilde{x}_i & := \begin{cases}x_i & (i=1, \cdots, t), \\
0 & (i=t+1, \cdots, a_1+1).\end{cases} \\
\end{aligned}
\end{equation*}
Then,\ we have
\begin{equation*}
(x_1+ \cdots+x_{a_1+1})-(\widetilde{x}_1+ \cdots+\widetilde{x}_{a_1+1}) < (a_1+1)C_5.
\end{equation*}
Note that $\widetilde{x}_1+ \cdots+\widetilde{x}_{a_1+1} \in [0, R) \cap (a_1+1)\widetilde{S_1}$.\ Thus, the maximality of $\theta$ implies that
\begin{equation*}
\begin{aligned}
\theta((a_1+1)\widetilde{S_1};R) &\geq \widetilde{x}_1+ \cdots+\widetilde{x}_{a_1+1} > x_1+ \cdots+x_{a_1+1}-(a_1+1)C_5 \\
&=\theta((a_1+1)S_1;R)-(a_1+1)C_5,
\end{aligned}
\end{equation*}
which leads to the first inequality of (\ref{eq:theta3}).\ By (\ref{eq:theta3}),\ Condition (2') and Condition (3), we get for any sufficiently large $R$ that
\begin{equation*}
\begin{aligned}
R-\theta((a_1+1)\widetilde{S_1};R) &< R-\theta((a_1+1)S_1;R)+(a_1+1)C_5 \\ 
&<\frac{R}{\lambda(S_1;R)^{a_1}\lambda(S_2;R)^{a_2+1}}+(a_1+1)C_5 \\ 
&\ll_{a_1,a_2}\frac{R}{\lambda(\widetilde{S_1};R)^{a_1}\lambda(\widetilde{S_2};R)^{a_2+1}}+(a_1+1)C_5 \\ 
&\ll_{a_1,a_2}\frac{R}{\lambda(\widetilde{S_1};R)^{a_1}\lambda(\widetilde{S_2};R)^{a_2+1}} \\
&=o\left(
\frac{R}{\lambda(\widetilde{S_1};R)^{a_1}\lambda(\widetilde{S_2};R)^{a_2}}
\right)
\end{aligned}
\end{equation*}
as $R$ tends to infinity, which implies Condition (3'). \par 
Since $\displaystyle \lim_{n\to\infty} n\un{\lambda}(n)^{-\be}=\infty $ by Condition (2'), we see that 
\[\mathcal{F}:= \{ N \in \mathbb{N} \mid n\un{\lambda}(n)^{-\be} \leq N\un{\lambda}(N)^{-\be} \ \text{for any} \ 0 \leq n \leq N \}\]
is an infinite set. 
Putting 
\[\lambda(N):=\max\{\lambda_1(N), \lambda_2(N)\},\]
we see for any positive real number $\varepsilon$ that 
\[\lambda(N)=o(N^{\varepsilon})\]
as $N$ tends to infinity from Condition (2').
In what follows, we take a sufficiently large element $N\in\mathcal{F}$. Condition (2') implies that 
\begin{equation}\label{eq:lambda1vsN}
N^\mathcal{C}>24C_7(1+C_1) \lambda(N)^{2D-1},
\end{equation}
where
\begin{equation*}
C_7:= 1+\operatorname{Card}  (\Gamma_1).
\end{equation*}

Define $W_2=W_2(N)$ by 
\begin{equation*}
W_2:= [0,N) \cap \left(\displaystyle \bigcup_{{\bm k}=(k_1, k_2,k_3) \in \Gamma_2} (k_1+k_2)\widetilde{S_{1}}+k_3\widetilde{S_{2}}\right).
\end{equation*}
Observe that if $N$ is sufficiently large, then $0,C_2'(\be')\in W_2$ because $[0,N) \cap \widetilde{S_2}\subset W_2$ by $(0,0,1)\in \Gamma_2$. 
For each $\bk=(k_1,k_2,k_3)\in \Gamma_2$, 
\begin{align*}
    \text{Card} \left([0,N) \cap \big((k_1+k_2)\wi{S_1}+k_3\wi{S_2}\big)
    \right)
    \leq
    \lambda_1(N)^{k_1+k_2} \lambda_2(N)^{k_3}=o\big(\un{\lambda}(N)^{\be}\big)
\end{align*}
as $N$ tends to infinity because $k_1+k_2 \leq g_1+g_2-1$ and $k_3<1+D$. Thus, we obtain 
\begin{equation*} 
\operatorname{Card}(W_2) \leq \sum_{{\bm k} \in \Gamma_2}\lambda_1(N)^{k_1+k_2}\lambda_2(N)^{k_3}\leq \dfrac{1}{32}{\underline{\lambda}(N)^{{\bm e}}}
\end{equation*}
for any sufficiently large $N\in\mathcal{F}$. 
Hence, we may write $W_2$ as follows:
\begin{equation}\label{eq:defW2}
W_2=\left \{0=r_{1}<r_{2}<\cdots<r_{\tau} \right \}, \quad \tau \leq \frac{1}{32}\underline{\lambda}(N)^{{\bm e}}.
\end{equation} 
For convenience,\ we set $r_{\tau+1} = N$, and
$$
\begin{aligned}
\mathcal{K}:=\left\{K=K(k)=\left[r_{k}, r_{1+k}\right) \mid 1 \leq k \leq \tau\right\} .
\end{aligned}
$$
Then $[0,N)$ can be represented as follows:
\begin{equation*} 
[0,N) = \displaystyle \bigcup_{K \in \mathcal{K}}K.
\end{equation*}
For an interval $I=[a,b)\subset\mathbb{R}$, define $|I|=b-a$. Then
\begin{equation} \label{eq:sumK}
\sum_{K \in \mathcal{K}}|K|=N.
\end{equation}
Moreover, define
$$
\begin{aligned}
& \mathcal{K}_{1} :=\left\{K \in \mathcal{K} \mid |K| \geq 16 N \underline{\lambda}(N)^{-{\bm{e}}}\right\}, \\
& \mathcal{K}_{2} :=\left\{K \in \mathcal{K}_{1} \mid K \subset\left[C_{2}'(\bm{e}'), N\right)\right\} .
\end{aligned}
$$
\begin{lem}\label{lem:W2} 
Let $N \in \mathcal{F}$ be sufficiently large. 
Then
\begin{equation*}
\sum_{K \in \mathcal{K}_{1}}|K| \geq \dfrac{N}{2}, \quad \sum_{K \in \mathcal{K}_{2}}|K| \geq \dfrac{N}{3}.
\end{equation*}
In particular, $\K_1$ and $\K_2$ are nonempty. 
\end{lem}

\begin{proof}
From (\ref{eq:defW2}) and (\ref{eq:sumK}), we see 
\begin{equation*}
\begin{split}
\sum_{K \in \mathcal{K}_{1}}|K|&=\sum_{K \in \mathcal{K}}|K|-\sum_{K \in \mathcal{K} \backslash \mathcal{K}_{1}}|K|  \geq N-\sum_{K \in \mathcal{K} \backslash \mathcal{K}_{1}} 16N \underline{\lambda}(N)^{-\bm{e}} \\ 
&\geq N-\tau \cdot 16 N \underline{\lambda}(N)^{-\bm{e}} \geq \dfrac{N}{2}.
\end{split}
\end{equation*}
Choose fixed elements $N_0, N_1\in S_2$ ($N_0 < N_1$) with $N_i > C_2'(\bm{e}')$ for $i=0,1$. 
In what follows, we assume that $N>6 N_1$, and so $N_0,N_1\in W_2$. 
Then there exists some $k_0 = k_0(N) \in \{1, 2, \ldots, \tau \}$ such that
\begin{equation*}
r_{k_{0}}=N_{0}, \quad r_{1+k_{0}} \leq N_{1}.
\end{equation*}
Putting $K(k)=[r_{k}, r_{1+k})$, we get 
\begin{align*}
\sum_{K \in \mathcal{K}_{2}}|K| &= \sum_{K \in \mathcal{K}_{1}}|K|-\sum_{K \in \mathcal{K}_{1} \backslash \mathcal{K}_{2}}|K|  \geq \sum_{K \in \mathcal{K}_{1}}|K|-\sum_{k=1}^{k_{0}}|K(k)| \\ &= \dfrac{N}{2}-(r_{k_{0}+1}-r_{1})  
\geq \dfrac{N}{2}-N_{1} \geq \dfrac{N}{3}.
\end{align*}

\end{proof}
\noindent Next, define $W_1=W_1(N)$ by 
\begin{align}\label{eq:defw1}
W_1:= [0,N) \cap \left(\displaystyle \bigcup_{\bm{k}=(k_1, k_2,k_3) \in \Gamma_1} (k_1+k_2)\widetilde{S_{1}}+k_3\widetilde{S_{2}}\right).
\end{align}
As in the estimate of $\operatorname{Card} (W_2)$, we get by the definition of $\Gamma_1$ and $\bk_1$ that 
\begin{align*}
\operatorname{Card}(W_1) &\leq \sum_{\bm{k} \in \Gamma_1}\lambda_1(N)^{k_1+k_2}\lambda_2(N)^{k_3} 
\leq \sum_{\bm{k} \in \Gamma_1}\lambda_1(N)^{g_1+g_2}\lambda_2(N)^{g_3-1}\\
&= \sum_{\bm{k} \in \Gamma_1} \underline{\lambda}(N)^{{\bm k}_1} <  C_7\underline{\lambda}(N)^{{\bm k}_1}.
\end{align*}
Thus, $W_1$ can be written as follows:
\begin{equation} \label{eq:defW1}
W_1=\left \{0=q_{1}<q_{2}<\cdots<q_{\mu} \right \}, \quad \mu \leq C_7\underline{\lambda}(N)^{\bm{k}_1}
\end{equation}
because $0 \in W_1$ by $0 \in \widetilde{S_1} \cap \widetilde{S_2}$.\ For convenience, we set $q_{\mu+1} = N$,\ and
\begin{equation*}
\mathcal{J}:=\left\{J=J(j)=\left[q_{j}, q_{1+j}\right) \mid 1 \leq j \leq \mu\right\}.
\end{equation*}
Then $[0,N)$ can be represented as follows:
\begin{equation*}
[0,N) = \displaystyle \bigcup_{J \in \mathcal{J}}J.
\end{equation*}
Then we have
\begin{equation*} 
\sum_{J \in \mathcal{J}}|J|=N .
\end{equation*}
Moreover, define
$$
\begin{aligned}
& \mathcal{J}_{1} :=\{J \in \mathcal{J} \mid \text{there exists}\ K \in \mathcal{K} \ \text{such that} \ J \subset K \}, \\
& \mathcal{J}_{2} :=\left\{J \in \mathcal{J}_{1} \relmiddle| |J| \geq \frac{1}{12 C_{7}} N \underline{\lambda}(N)^{-\bm{k}_{1}}\right\}.
\end{aligned}
$$
\begin{lem} \label{lem:W1} 
Let $N \in \mathcal{F}$ be sufficiently large. Then 
\begin{equation*}
\sum_{J \in \mathcal{J}_{1}}|J| \geq \dfrac{N}{6}, \quad \sum_{J \in \mathcal{J}_{2}}|J| \geq \dfrac{N}{12}.
\end{equation*}
In particular, $\mathcal{J}_{1}$ and $\mathcal{J}_{2}$ are nonempty.
\end{lem}
\noindent To prove Lemma \ref{lem:W1},\ we show Lemma \ref{lem:theta}. 

\begin{lem} \label{lem:theta}
Let $M$ and $E$ be any positive real numbers with
\begin{equation*}
M \geq C_{2}'(\bm{e}') \ \text{and} \ E \geq 4 M \underline{\lambda}(M)^{-\bm{e}}.
\end{equation*}
Then
\begin{equation*}
M+\dfrac{E}{2}<\theta(M+E).
\end{equation*}
\end{lem}

\begin{proof}[Proof of Lemma \ref{lem:theta}]
By taking $C_2'(\bm{e}')$ sufficiently large,\ we may assume that $\lambda_{2}(M) > 4$.\ Then we have 
\begin{equation*}
\dfrac{E}{4} \geq M \underline{\lambda}(M)^{-\bm{e}}, \quad \dfrac{E}{4}>E \underline{\lambda}(M)^{-\bm{e}}.
\end{equation*}
Therefore,
\begin{equation} \label{eq:4.4.1}
\dfrac{E}{2} >(M+E) \underline{\lambda}(M)^{-\bm{e}} \geq (M+E) \underline{\lambda}(M+E)^{-\bm{e}}.
\end{equation}
Note that $M+E \geq C_{2}'(\bm{e}')$.\ Using (\ref{eq:4.4.1}) and Condition\,(3') with $(a_1,a_2)=\be',\ R=M+E$,\ we obtain 
\begin{equation*}
\begin{split}
M+E-\theta(M+E) &= M+E-\theta(((g_1+g_2-1)+1) \widetilde{S_{1}}; M+E) \\
                      &< (M+E) \underline{\lambda}(M+E)^{-\bm{e}} < \dfrac{E}{2}. 
\end{split}
\end{equation*}
Thus we obtain $M+\dfrac{E}{2}<\theta(M+E)$.
\end{proof}
\begin{proof} [Proof of Lemma \ref{lem:W1}]
Take any $K=\left[r_{k}, r_{1+k}\right) \in \mathcal{K}_{2}$.\ Consider $J\in \J$ such that $J\subset K$.  
We have by the definition of $\mathcal{K}_2$ that $C_{2}'(\bm{e}') \leq r_{k}<r_{1+k} \leq N$ and $|K| \geq 16 N \underline{\lambda}(N)^{-\bm{e}}$. 
Moreover,\ by the definition of $\mathcal{F}$,
\begin{equation} \label{eq:L1}
\dfrac{|K|}{4} \geq 4 N \underline{\lambda}(N)^{-\bm{e}} \geq 4 r_{1+k} \underline{\lambda}(r_{1+k})^{-\bm{e}} \geq r_{1+k} \underline{\lambda}\left(r_{1+k}\right)^{-\bm{e}}.
\end{equation}
Using (\ref{eq:L1}) and Condition\,(3') with $(a_1,a_2)=\be',\ R=r_{1+k}$,\ we obtain
\begin{equation*}
r_{1+k}>\theta\left(r_{1+k}\right) > r_{1+k}-r_{1+k}\underline{\lambda}\left(r_{1+k}\right)^{-\bm{e}} \geq r_{1+k}-\dfrac{|K|}{4}.
\end{equation*}
Note by $(g_1+g_2,0,0)\in \Gamma_1$ that 
\[\theta\left(r_{1+k}\right) \in (g_1+g_2)\widetilde{S_{1}}+0\wi{S_2}\subset W_1. \] 
On the other hand,\ by the definition of $\mathcal{K}_{2}$ and $\mathcal{F}$,\ we have
\begin{equation*}
\dfrac{|K|}{4} \geq 4 N \underline{\lambda}(N)^{-\bm{e}} \geq 4 r_{k} \underline{\lambda}\left(r_{k}\right)^{-\bm{e}}.
\end{equation*}
Thus, we can apply Lemma \ref{lem:theta} with $M=r_{k}$ and $E=\dfrac{|K|}{4}$. 
Hence, we obtain
\begin{equation*}
r_k < r_{k}+\dfrac{|K|}{8}<\theta\left(r_{k}+\dfrac{|K|}{4}\right)<r_{k}+\dfrac{|K|}{4}\big(<\theta(r_{1+k})\big),
\end{equation*}
and $\theta\left(r_{k}+\dfrac{|K|}{4}\right) \in (g_1+g_2)\widetilde{S_{1}}\subset W_1$.
Defining 
\begin{equation*}
\begin{split}
& \beta(K):=\min \left\{n \in \mathbb{N} \mid r_{k}<n \ \text{and} \ n \in W_1 \right\}, \\
& \gamma(K):=\max \left\{n \in \mathbb{N} \mid n<r_{1+k} \ \text{and} \ n \in W_1 \right\},
\end{split}
\end{equation*}
we have the following inequalities: 
\begin{align} \label{eq:4.8.1}
&\beta(K) \leq \theta\left(r_{k}+\dfrac{|K|}{4}\right)<r_{k}+\dfrac{|K|}{4}.\\
\label{eq:4.8.2}
&\gamma(K) \geq \theta\left(r_{1+k}\right)>r_{1+k}-\dfrac{|K|}{4}.
\end{align}
Also,\ we have 
\begin{equation*}
\bigcup_{J \in \mathcal{J}, J \subset K}J=[\beta(K),\gamma(K)), 
\end{equation*}
and so 
\begin{equation} \label{eq:4.8.3}
\sum_{J \in \mathcal{J}, J \subset K}|J|=\gamma(K)-\beta(K).
\end{equation}
Therefore, by combining (\ref{eq:4.8.1}), (\ref{eq:4.8.2}) and (\ref{eq:4.8.3}),
\begin{equation}\label{eq:4.8.4}
\begin{split}
\sum_{J \in \mathcal{J}, J \subset K}|J| &= \gamma(K)-\beta(K)  \geq \left(r_{1+k}-\dfrac{|K|}{4}\right)-\left(r_{k}+\dfrac{|K|}{4}\right) \\
&=|K|-\dfrac{|K|}{2}=\dfrac{|K|}{2}. 
\end{split}
\end{equation}
By Lemma \ref{lem:W2},\ (\ref{eq:4.8.4}) and the definition of $\mathcal{J}_1$,\ we obtain
\begin{equation*}
\sum_{J \in \mathcal{J}_{1}}|J| \geq \sum_{K \in \mathcal{K}_{2}} \sum_{J \in \mathcal{J}, J \subset K}|J| \geq \frac{1}{2} \sum_{K \in \mathcal{K}_{2}}|K| \geq \frac{N}{6}.
\end{equation*}
Finally,\ we evaluate $\displaystyle \sum_{J \in \mathcal{J}_{2}}|J|$.\ By (\ref{eq:defW1}) and the first inequality of Lemma \ref{lem:W1},\ we obtain
\begin{equation*}
\sum_{J \in \mathcal{J}_{2}}|J|=\sum_{J \in \mathcal{J}_{1}}|J|-\sum_{J \in \mathcal{J}_{1} \backslash \mathcal{J}_{2}}|J| \geq \frac{1}{6} N-\mu \cdot \frac{1}{12 C_{7}} N \underline{\lambda}(N)^{-\bm{k}_{1}} \geq \frac{N}{12}.
\end{equation*}
\end{proof}
\noindent Next, we take an interval $J_1 = [q_y,q_{1+y})\in \J_2$ satisfying 
\begin{equation*}
|J_1| = \max_{J \in \mathcal{J}_2} |J|.
\end{equation*}
By the definition of $\J_2$, it follows from $[0,N)\cap V=W_1\cup W_2$ that 
\begin{align}
    \label{eq:interJ1}
    (q_y,q_{1+y})\cap V=\emptyset .
\end{align}
We give a lower bound for $|J_1|$.
By Lemma \ref{lem:W1},\ we get
\begin{align}
\frac{N}{12} &\leq \sum_{J \in \mathcal{J}_2} |J|  \leq \mu \max_{J \in \mathcal{J}_2} |J|  
\leq |J_1| C_7 \underline{\lambda} (N)^{{\bm k}_1}\nonumber\\
&\leq |J_1| C_7 \lambda(N)^{g_1+g_2+g_3-1} 
\leq |J_1| C_7 \lambda(N)^{D-1}. \label{eq:J1vsN}
\end{align}
Moreover,\ by (\ref{eq:lambda1vsN}),
\begin{equation*}
\lambda(N)^{D-1}=\lambda(N)^{2D-1}\lambda(N)^{-D}< \frac{1}{24C_7(1+C_1)}N^\mathcal{C}\lambda(N)^{-D}.
\end{equation*}
Therefore,\ for any sufficiently large $N\in \mathcal{F}$,
\begin{equation} \label{eq:j1length}
N^{1-\mathcal{C}} \leq |J_1| \frac{1}{2(1+C_1)} \lambda(N)^{-D} <\frac{|J_1|}{1+C_1}.
\end{equation}
In particular, if $N\in \mathcal{F}$ is sufficiently large, then there exists an integer $\theta_0$ satisfying
\begin{equation}\label{eq:theta0}
\theta_0 \in \bigg[ \frac{|J_1|}{1+C_1}, \frac{C_1|J_1|}{1+C_1} \bigg) \cap \widetilde{S_{2}}.
\end{equation}
by Condition\,(1').\ Now we define $M$ and  $J_2$ as
\begin{align*}
M&:= q_y+ \theta_0 \in \bigg[ q_y+\frac{|J_1|}{1+C_1}, q_y+\frac{C_1|J_1|}{1+C_1} \bigg),\\
J_2 &:= [q_y,M) \subset [q_y, q_{1+y})=J_1, 
\end{align*}
respectively.\ Note that there exists ${\bm k}=(k_1,k_2,k_3) \in \Gamma_1$ satisfying 
\begin{equation*}
M = q_y+ \theta_0 \in \bigg((k_1+k_2) \widetilde{S_{1}}+ k_3 \widetilde{S_{2}} \bigg) + \widetilde{S_{2}} 
\subset (g_1+g_2)\widetilde{S_{1}}+g_3 \widetilde{S_{2}}.
\end{equation*}
Lemma \ref{lem:Minkowski} implies that 
\begin{equation}\label{eq:rhoGM}
\rho({\bm g} ; M) \geq 1.
\end{equation}

\begin{lem} \label{lem:YR0}
Suppose that $N \in \mathcal{F}$ is sufficiently large.\ Then for any $R \in J_2$,\
$$
Y_R>0.
$$
\end{lem}

\begin{proof}
We use induction on $R$.\ First we consider the case of $R=M-1$.\ Observe that
\begin{equation*}
\begin{aligned}
Y_{M-1} &= \sum_{{\bm k} \in \Gamma} B_{{\bm k}} \sum_{m=1}^\infty \beta^{-m} \rho({\bm k};m+M-1) \\
&=\sum_{m=1}^{q_{1+y}-M} \beta^{-m} \sum_{{\bm k} \in \Gamma} B_{{\bm k}}  \rho({\bm k};m+M-1) \\
&+ \sum_{{\bm k} \in \Gamma} B_{{\bm k}} \sum_{m=1+q_{1+y}-M}^\infty \beta^{-m} \rho({\bm k};m+M-1)\\
&=:S_1+S_2.
\end{aligned}
\end{equation*}
First, we give the upper bound for $|S_2|$.\ By (\ref{eq:j1length}) and (\ref{eq:theta0}),\ we get
\begin{equation} \label{eq:q1y}
\begin{aligned}
q_{1+y}-M &=q_{1+y}-q_y - \theta_0 =|J_1| - \theta_0 \\
&\geq |J_1| - \frac{C_1|J_1|}{1+C_1}=\frac{|J_1|}{1+C_1}>N^{1-\mathcal{C}}
\end{aligned}
\end{equation}
for sufficiently large $N \in \mathcal{F}$. Putting
\begin{equation*}
\zeta_n := \beta^{n^{1-2\mathcal{C}}},
\end{equation*}
we get by Lemma \ref{lem:logrho} for any ${\bm k} \in \Gamma$ that  
\begin{equation} \label{eq:zetaev}
\rho({\bm k} ; m) \ll \zeta_m.
\end{equation}
Observing by the mean value theorem and $0<2\mathcal{C}<1$ that  
\begin{equation*}
\lim_{m \to \infty}\big((m+1)^{1-2\mathcal{C}} - m^{1-2\mathcal{C}} \big) =0,
\end{equation*}
we have 
\begin{equation*}
\lim_{m \to \infty} \frac{\zeta_{m+1}}{\zeta_m} =1.
\end{equation*}
Therefore,\ for sufficiently large $m$,\ we obtain
\begin{equation} \label{eq:zeta3}
\frac{\zeta_{m+1}}{\zeta_m}<\frac{\beta +1}{2}.
\end{equation}
By 
\begin{align*}
&\lim_{N\to\infty}\log_{\beta}\left(
\beta^{-N^{1-\mathcal{C}}}\zeta_{2N}\beta^{2N^{1-2\mathcal{C}}}
\right)\\
&=\lim_{N\to\infty}\left(
-N^{1-\mathcal{C}}+(2N)^{1-2\mathcal{C}}+2N^{1-2\mathcal{C}}
\right)=-\infty,
\end{align*}
we see 
\begin{align}\label{eq:AAAB}
    \beta^{-N^{1-\mathcal{C}}}\zeta_{2N}=o\left(
    \beta^{-2N^{1-2\mathcal{C}}}
    \right)
\end{align}
as $N$ tends to infinity. 
Thus by (\ref{eq:q1y}), (\ref{eq:zetaev}), (\ref{eq:zeta3}) and (\ref{eq:AAAB}), 
\begin{equation} \label{eq:J2YR1}
\begin{aligned}
|S_2| &\leq \sum_{{\bm k} \in \Gamma} |B_{{\bm k}} | \sum_{m=1+q_{1+y}-M}^\infty \beta^{-m} \rho({\bm k};m+M-1) \\
& \ll \sum_{m=1+q_{1+y}-M}^\infty \beta^{-m} \zeta_{m+M-1} \\
& \leq \sum_{m> N^{1-\mathcal{C}}} \beta^{-m} \zeta_{m+N} \\ 
& = \sum_{m=0}^\infty  \beta^{- \big( m +\lfloor N^{1-\mathcal{C}} \rfloor +1\big)} \zeta_{m +\lfloor N^{1-\mathcal{C}} \rfloor +1+N} \\ 
& < \beta^{-\big(\lfloor N^{1-\mathcal{C}} \rfloor +1\big)} \zeta_{\lfloor N^{1-\mathcal{C}} \rfloor +1+N} \sum_{m=0}^\infty \beta^{-m}\bigg( \frac{\beta+1}{2} \bigg)^{m} \\
& \ll \beta^{-N^{1-\mathcal{C}}} \zeta_{2N} =o(1)
\end{aligned}
\end{equation}
as $N \in \mathcal{F}$ tends to infinity. \par 
Next, we give the lower bound for $S_1$.\ For any integer $m$ with $1\leq m \leq q_{1+y} -M$,\ we have
\begin{equation*}
q_y <M \leq m+M-1 \leq q_{1+y}-1 <q_{1+y},
\end{equation*}
so we see $m+M-1 \notin V$ by (\ref{eq:interJ1}). Therefore,\ by Lemma \ref{lem:V1} and (\ref{eq:rhoGM}),\ we obtain
\begin{equation} \label{eq:J2YR2}
\begin{aligned}
S_1 &\geq \sum_{m=1}^{q_{1+y}-M} \beta^{-m} \bigg( B_{\bm g} \rho({\bm g} ; m+M-1) -
\sum_{{\bm k} \in \Gamma \setminus \{ {\bm g} \}}|B_{\bm k}| \rho({\bm k} ; m+M-1) \bigg)\\
& \geq \frac{1}{2}B_{\bm g} \sum_{m=1}^{q_{1+y}-M}\beta^{-m}\rho({\bm g} ; m+M-1) \\
&\geq \frac{1}{2} \beta^{-1} B_{\bm g} \rho({\bm g};M) \geq \frac{1}{2} B_{\bm g} \beta^{-1} 
\end{aligned}
\end{equation}
Therefore, by (\ref{eq:J2YR1}) and (\ref{eq:J2YR2}),\ we get 
\begin{equation*}
Y_{M-1}>0
\end{equation*}
for sufficiently large $N \in \mathcal{F}$. \par
Next we assume that $Y_R>0$ for some $R\in (q_y,M)$.\ Then
\begin{equation*}
\begin{split}
Y_{R-1} &= \sum_{{\bm k} \in \Gamma} B_{{\bm k}} \sum_{m=1}^\infty \beta^{-m} \rho({\bm k};m+R-1) \\
&=\frac{1}{\beta}\sum_{{\bm k} \in \Gamma} B_{{\bm k}}  \rho({\bm k};R) 
+ \frac{1}{\beta}\sum_{{\bm k} \in \Gamma} B_{{\bm k}} \sum_{m=2}^\infty \beta^{-(m-1)} \rho({\bm k};m+R-1)\\
&=:S_1'+S_2'.
\end{split}
\end{equation*}
First, we give the lower bound for $S_1'$. 
Note by (\ref{eq:interJ1}) that $R \notin V$ for each $R \in (q_y,M)$. 
Lemma \ref{lem:V1} implies 
\begin{equation} \label{J2YRcase2-1}
S_1' \geq \frac{1}{\beta} \bigg( B_{\bm g} \rho({\bm g} ; R) -
\sum_{{\bm k} \in \Gamma \setminus \{ {\bm g} \}}|B_{\bm k}| \rho({\bm k} ; R) \bigg) \geq \frac{1}{2\beta} B_{\bm g}\rho({\bm g};R) \geq 0.
\end{equation}
Second,\ we give the lower bound for $S_2'$.\ By the induction hypothesis,\ we see that
\begin{equation}\label{J2YRcase2-2}
S_2' = \frac{1}{\beta}\sum_{{\bm k} \in \Gamma} B_{{\bm k}} \sum_{m=1}^\infty \beta^{-m} \rho({\bm k};m+R) =\frac{1}{\beta} Y_R>0.
\end{equation}
Hence,\ by (\ref{J2YRcase2-1}) and (\ref{J2YRcase2-2}),\ we get $Y_{R-1}>0$.
\end{proof}
Now we put
\begin{equation*}
J_2 \cap \big( (g_1+g_2)\widetilde{S_{1}} + g_3 \widetilde{S_{2}} \big) =: \{q_y=p_1 < p_2 <\cdots < p_\kappa \},
\end{equation*}
where
\begin{equation*}
\kappa \leq \lambda_1(N)^{g_1+g_2} \lambda_2(N)^{g_3} \leq \lambda(N)^D
\end{equation*}
for any sufficiently large $N\in \mathcal{F}$.
For convenience,\ set $p_{1+\kappa} := M$.\ Then there exists an integer $x$ with $1 \leq x \leq \kappa$ such that
\begin{equation*}
p_{1+x} -p_x = \max_{1 \leq i \leq \kappa } \{ p_{1+i} -p_i \}.
\end{equation*}
Put $I_1:= [p_x,p_{1+x})$.\ For any ${\bm k}=(k_1,k_2,k_3) \in \Gamma$,
\begin{equation}\label{eq:interI1}
(p_x,p_{1+x}) \cap \bigg((k_1+k_2) \widetilde{S_{1}} + k_3\widetilde{S_{2}} \bigg) = \emptyset.
\end{equation}
We give a lower bound for $|I_1|$.\ We define $\mathcal{I}$ as
\begin{equation*}
\mathcal{I}:= \{ I=I(i) = [p_i, p_{1+i}) \mid 1 \leq i \leq \kappa \}.
\end{equation*}
We see that
\begin{equation} \label{eq:J2ev}
|J_2| 
= \sum_{i=1}^\kappa |I(i)| \leq \kappa |I_1| \leq \lambda(N)^D|I_1|.
\end{equation}
Combining (\ref{eq:lambda1vsN}), (\ref{eq:J1vsN}), (\ref{eq:J2ev}), and the definition of $M$, 
we obtain
\begin{equation} \label{eq:I1length}
\begin{aligned}
|I_1| &\geq \frac{|J_2|}{\lambda(N)^D} = \frac{M-q_y}{\lambda(N)^D} \\
& \geq \frac{1}{\lambda(N)^D} \cdot \frac{|J_1|}{1+C_1} \\
& \geq \frac{N}{12C_7(1+C_1)\lambda(N)^{2D-1}} \\
& > 2N^{1-\mathcal{C}}.
\end{aligned}
\end{equation}
\begin{lem}\label{lem:I1YR}
Suppose that $N \in \mathcal{F}$ is sufficiently large.\ Then, for any nonnegative integer $R$ with
$$p_x \leq R < p_{x} + \frac{1}{2}|I_1|,$$
we have
$$0<Y_R<\beta^{-R^{1-2\mathcal{C}}}.$$
\end{lem}
\begin{proof}
By Lemma \ref{lem:YR0},\ it suffices to show that $|Y_R|<\beta^{-R^{1-2\mathcal{C}}}$. 
By (\ref{eq:I1length}),\ we see that
\begin{equation}\label{eq:last1}
p_{1+x} -R > p_{1+x} -p_x -\frac{1}{2}|I_1| = \frac{1}{2}|I_1| > N^{1-\mathcal{C}}.
\end{equation}
Let $m$ be an integer with $1\leq m \leq p_{1+x} -R-1$.
Since $m+R \in (p_x,p_{1+x})$, we see by (\ref{eq:interI1}) that 
\begin{equation*}
\rho({\bm k};m+R) =0
\end{equation*}
for any ${\bm k} \in \Gamma$.\ Thus by (\ref{eq:zetaev}),
\begin{align}
|Y_R| &\leq \sum_{{\bm k} \in \Gamma} |B_{\bm k}| \sum_{m=p_{1+x} -R}^\infty \beta^{-m} \rho({\bm k};m+R) \nonumber\\
& \ll \sum_{{\bm k} \in \Gamma} |B_{\bm k}| \sum_{m=p_{1+x} -R}^\infty \beta^{-m} \zeta_{m+R} \nonumber\\
& \ll \sum_{m=p_{1+x} -R}^\infty \beta^{-m} \zeta_{m+N}.\label{eq:last2}
\end{align}
Combining (\ref{eq:zeta3}), (\ref{eq:AAAB}), (\ref{eq:last1}) and (\ref{eq:last2}), we deduce that 
\begin{equation*}
|Y_R| \ll \sum_{m>N^{1-\mathcal{C}}} \beta^{-m} \zeta_{m+N} 
\ll \beta^{-N^{1-\mathcal{C}}}\zeta_{2N}\sum_{m=0}^{\infty}\beta^{-m}\left(\frac{\beta+1}2\right)^m
\ll \beta^{-2N^{1-2\mathcal{C}}}.
\end{equation*}
Hence, there exists a positive constant $C_8$ such that
\begin{equation*}
|Y_R| \leq C_8\beta^{-2N^{1-2\mathcal{C}}} = \frac{C_8}{\beta^{N^{1-2\mathcal{C}}}}\beta^{-N^{1-2\mathcal{C}}} < \beta^{-N^{1-2\mathcal{C}}} \leq \beta^{-R^{1-2\mathcal{C}}} 
\end{equation*}
for sufficiently large $N\in \mathcal{F}$,\ which implies Lemma \ref{lem:I1YR}.
\end{proof}

Finally, we prove Theorem \ref{mainThm}. 

\section*{Acknowledgments}
The first author was supported by the JSPS KAKENHI Grant Number 24K06641. The second and third authors were supported by JST SPRING, Grant Number JPMJSP2124.

\bibliographystyle{plain}
\bibliography{myref_revised}
\end{document}